\documentclass[a4paper]{amsart}

\usepackage[plainpages=false]{hyperref}
\usepackage[english]{babel}
\usepackage[latin1]{inputenc}
\usepackage{verbatim}
\usepackage[arrow, matrix, curve]{xy}
\usepackage{amsmath,amssymb,amsfonts,amsthm,amsopn}
\usepackage{nicefrac}
\usepackage{fancyhdr}

\title{A Note on Extensions of $\mathbb{Q}^{tr}$}
\author{Lukas Pottmeyer}
\address{Technische Universit\"at Darmstadt; Schlo{\ss}gartenstr. 7, 64289 Darmstadt}
\email{pottmeyer@mathematik.tu-darmstadt.de}
\date{\today}

\DeclareMathOperator{\Gal}{Gal}

\newtheorem{theorem}[]{Theorem}

\newtheorem{corollary}[]{Corollary}

\newtheorem{lemma}[]{Lemma}
\newtheorem*{definition}{Definition}
\newtheorem{example}[]{Example}
\newtheorem{question}[]{Question}
\newtheorem*{remark}{Remark}

\begin{document}

\selectlanguage{english}

\begin{abstract}
In this note we investigate the behaviour of the absolute logarithmic Weil-height $h$ on extensions of the field $\mathbb{Q}^{tr}$ of totally real numbers. It is known that there is a gap between $0$ and the next smallest value of $h$ on $\mathbb{Q}^{tr}$, whereas in $\mathbb{Q}^{tr}(i)$ there are elements of arbitrarily small positive height. We prove that all elements of small height in any finite extension of $\mathbb{Q}^{tr}$ already lie in $\mathbb{Q}^{tr}(i)$.
This leads to a positive answer to the question of Amoroso, David and Zannier, if there exists a pseudo algebraically closed field with the mentioned height gap. 
\end{abstract}

\maketitle

\section{Introduction}
 
We assume throughout the paper that all algebraic extensions of $\mathbb{Q}$ are contained in a fixed algebraic closure $\overline{\mathbb{Q}}\subseteq \mathbb{C}$.
For a definition and basic properties of the height $h$ we refer to the first Chapter of \cite{BG}.
For a number field $K$ we define $$K^{tr}=\{\alpha \in \overline{\mathbb{Q}}\vert \sigma(\alpha) \in \mathbb{R} ~\forall \sigma \in \Gal(\overline{\mathbb{Q}}/K)\}\quad.$$ 
A theorem of Schinzel \cite{Sch73}, Theorem 2, implies $h(\alpha)\geq \frac{1}{2}\log\left(\frac{1+\sqrt{5}}{2}\right)$ for all $\alpha \in \mathbb{Q}^{tr}(i)^*$ which do not lie on the unit circle.
In this note we will generalize this result in the following way.

\begin{theorem}\label{finextensions}
Let $K$ be a number field and $F$ a finite extension of $K^{tr}$. There is an effective positive constant $c_F$ such that $h(\alpha)\geq c_F$ for all $\alpha \in F^*$ with at least one $K$-conjugate not on the unit circle; i.e. for all $\alpha\in F^*$ such that $\vert\sigma(\alpha)\vert \neq 1$ for some $\sigma \in \Gal(\overline{\mathbb{Q}}/K)$.
\end{theorem}

Using the notation of Bombieri and Zannier from \cite{BZ01}, we say that a field $F\subseteq \overline{\mathbb{Q}}$ has the Bogomolov property if and only if there is a positive constant $c$ such that the height $h$ on elements of $F$ is either zero or bounded from below by $c$. By Northcott's theorem every number field has the Bogomolov property.
Hence, the interesting cases are those where the Bogomolov property is satisfied by extensions of infinite degree over $\mathbb{Q}$. Schinzel's result implies that $\mathbb{Q}^{tr}$ has the Bogomolov property. More fields with the Bogomolov property and infinite degree over $\mathbb{Q}$, are maximal abelian \cite{AZ00} and maximal totally $p$-adic \cite{BZ01} extensions of a given number field, and fields of the form $\mathbb{Q}(E_{\rm tors})$ where $E$ is an elliptic curve defined over $\mathbb{Q}$ \cite{Ha13}.

The Bogomolov property is in general not preserved under finite extensions. This is due to the fact that there are elements in $\mathbb{Q}^{tr}(i)$ of arbitrarily small positive height (in particular Schinzel's bound does not hold if we drop the assumption $\vert \alpha \vert \neq 1$). This follows from an early result of May \cite{Ma72}, Example 1, where he studies the multiplicative groups of subfields of $\mathbb{Q}^{tr}$ and $\mathbb{Q}^{tr}(i)$. The first explicit proof of the failure of the Bogomolov property in the regarded extension is due to Amoroso and Nuccio in \cite{AN07}. A proof using dynamical height function can be found in \cite{Po13}, Remark 4.5.
An explicit sequence of elements in $\mathbb{Q}^{tr}(i)$ of height tending to zero is given by taking $n$-th roots of the element $\frac{2+i}{2-i}$ (see for example \cite{ADZ11}, Example 5.3, or Lemma \ref{lemma:1} below).

The extension $\mathbb{Q}^{tr}(i)/\mathbb{Q}^{tr}$ is essentially the only known example of a finite extension where the Bogomolov property is not preserved; i.e. in all examples of such extensions $L/F$ there is a sequence of points of small height contained in $L\cap\mathbb{Q}^{tr}(i)$. In particular, it is open whether these are the only possible examples of finite extensions not preserving the Bogomolov property. Theorem \ref{finextensions} implies that a finite extension of $K^{tr}$ has the Bogomolov property if and only if it does not contain the element $i$ (see Corollary \ref{cor:2}).

Since the Bogomolov property is in general not preserved under finite extensions, one can ask whether or not it is preserved for other special field extensions; for instance its Galois closure.

\begin{question}\label{question:2}
Is there a field which satisfies the Bogomolov property but the Galois closure over $\mathbb{Q}$ of this field does not?
\end{question}

In Example \ref{answer:2} we will construct an explicit field to provide a positive answer to this question. By a result of Widmer \cite{Wi11}, Corollary 2, it is known that the Northcott property of a field is not preserved under taking the Galois closure of this field. Here a field $F\subseteq \overline{\mathbb{Q}}$ has the Northcott property if every set of elements in $F$ with bounded height is finite. In particular, the Northcott property implies the Bogomolov property. It would be interesting to know whether the Galois closure of a field with the Northcott property necessarily satisfies the Bogomolov property.

As a further application of our result we will provide a positive answer to a question of Amoroso, David and Zannier concerning the relation between fields with the Bogomolov property and pseudo algebraically closed fields. 

\begin{definition}\rm \label{PAC}
Let $F$ be a field with algebraic closure $\overline{F}$. Then $F$ is called pseudo algebraically closed (PAC) if and only if every absolutely irreducible variety defined over $F$ has an $F$-rational point.
\end{definition}

If $F\subseteq \overline{\mathbb{Q}}$, then by \cite{FJ}, Theorem 11.2.3, $F$ is PAC if and only if for every polynomial $f(x,y)\in\mathbb{Q}[x,y]$ which is irreducible in $\overline{\mathbb{Q}}[x,y]$ there are $\alpha,\beta\in F$ with $f(\alpha,\beta)=0$. 

Obviously every algebraically closed field is PAC. Moreover every algebraic extension of a PAC field is again PAC. Hence, in sharp contrast to the Bogomolov property which is obviously inherited to subfields, the property of a field to be PAC indicates that a field is \emph{large}. This vague observation leads to the following question.

\begin{question}[\cite{ADZ11}, Problem 6.1]\label{question:1}
Is there a PAC field which satisfies the Bogomolov property?
\end{question}

The authors of \cite{ADZ11} raise some doubts towards a positive answer. In particular they disprove the Bogomolov property for some PAC fields contained in $\overline{\mathbb{Q}}$. However, surprisingly there are PAC fields satisfying the Bogomolov property. In Example \ref{answer:1} we will explicitly construct such a field as a totally imaginary finite extension of $\mathbb{Q}^{tr}$. This example comes from a deep result of Pop \cite{Pop96} concerning the arithmetic of the field $\mathbb{Q}^{tr}$.
For detailed information about PAC fields we refer to \cite{FJ}, Chapter 11.

In the next section we will study the field $K^{tr}$ and prove Theorem \ref{finextensions}. The proof of this result is elementary and the only non standard contribution is an effective lower height bound in terms of the ratio of the real embeddings of an algebraic number due to Garza \cite{Ga07} (see also \cite{Ho11} for a short proof of Garza's result).

In the last section of this note we will present explicit positive answers to Questions \ref{question:2} and \ref{question:1}.

\section{Proof of Theorem \ref{finextensions}}

Let throughout this section $K$ be any number field. In order to prove the main observation in this note, we need some more notation. For an algebraic number $\alpha$ we denote by $r_{\alpha,K}$ the number of real $K$-conjugates of $\alpha$, and we set $R_{\alpha,K}=\frac{r_{\alpha,K}}{[K(\alpha):K]}$. We have the following inequality.

\begin{equation}\label{eq:1}
[K:\mathbb{Q}]R_{\alpha,\mathbb{Q}} \geq \frac{r_{\alpha,\mathbb{Q}}[K:\mathbb{Q}]}{[\mathbb{Q}(\alpha):\mathbb{Q}]} \geq \frac{r_{\alpha,K}[K:\mathbb{Q}]}{[\mathbb{Q}(\alpha):\mathbb{Q}]} = \frac{r_{\alpha,K}}{[K(\alpha):K]}\frac{[K(\alpha):\mathbb{Q}]}{[\mathbb{Q}(\alpha):\mathbb{Q}]}\geq R_{\alpha,K}.
\end{equation}

The next two lemmas provide some basic information on the field $K^{tr}$ which are used in the proof of Theorem \ref{finextensions}.

\begin{lemma}\label{wlog}
It is $K^{tr}=(K\cap\mathbb{R})^{tr}$ and $\mathbb{Q}^{tr} \subseteq K^{tr} \subseteq \mathbb{R}$.
\end{lemma}

\begin{proof} It follows straight from the definition that $(K\cap\mathbb{R})^{tr}\subseteq K^{tr}$. For $\alpha\in \overline{\mathbb{Q}}$ let $f$ be its minimal polynomial defined over $K$. Then $\alpha$ is in $K^{tr}$ if and only if all roots of $f$ are real. If this is satisfied then all coefficients of $f$ are real and hence $f$ is defined over $(K\cap \mathbb{R})$. Thus, $f$ is also the minimal polynomial of $\alpha$ over $K\cap \mathbb{R}$. This immediately implies $\alpha \in (K\cap\mathbb{R})^{tr}$ and hence $K^{tr} = (K\cap \mathbb{R})^{tr}$. The second statement of the lemma is trivial.
\end{proof}

It is well known that $\mathbb{Q}^{tr}(i)$ is the maximal CM-field and hence generated by all elements having all $\mathbb{Q}$-conjugates on the unit circle (see \cite{AN07}, Proposition 2.3, or \cite{BL78}, Theorem 1). This generalizes one by one if we replace $\mathbb{Q}$ by any number field $K$.

\begin{lemma}\label{lemma:1}
Let $\alpha \in \overline{\mathbb{Q}}\setminus \{\pm1\}$ with $\vert \sigma(\alpha)\vert=1$ for all $\sigma \in \Gal(\overline{\mathbb{Q}}/K)$. Then $K^{tr}(\alpha)=K^{tr}(i)$.
\end{lemma}

\begin{proof}
Let $\alpha$ be as described and $\sigma \in \Gal(\overline{\mathbb{Q}}/K)$ arbitrary.
The assumption on $\alpha$ implies $\sigma(\alpha)^{-1}=\overline{\sigma(\alpha)}$, 
where the bar denotes complex conjugation. Therefore, $\sigma(\frac{1}{2}(\alpha + \alpha^{-1}))= \frac{1}{2}(\sigma(\alpha) + \sigma(\alpha)^{-1}) \in \mathbb{R}$. 
This means that $Re(\alpha)$, the real part of $\alpha$, is in $K^{tr}$. In the same fashion we get $\sigma(\frac{i}{2}(\alpha - \alpha^{-1})) = \pm\frac{i}{2}(\sigma(\alpha) - \sigma(\alpha)^{-1}) \in \mathbb{R}$, 
which is equivalent to the statement that $Im(\alpha)$, the imaginary part of $\alpha$, is in $K^{tr}$, too. By assumption $\alpha \not\in \mathbb{R}$ and hence $Im(\alpha)\neq 0$. This yields $i = \frac{\alpha - Re(\alpha)}{Im(\alpha)}\in K^{tr}(\alpha)$ and of course $\alpha \in K^{tr}(i)$.
\end{proof}

\begin{proof}[Proof of Theorem \ref{finextensions}]
By Lemma \ref{wlog} we can and will assume without loss of generality that $K$ is a real number field. Note that this implies $K \subseteq K^{tr}$.
 
Let us assume first that $F$ has at least one real $K$-embed\-ding. This means that $F=K^{tr}(\alpha)$ for an $\alpha \in \overline{\mathbb{Q}}$ such that there is a $\sigma \in \Gal(\overline{\mathbb{Q}}/K)$ with $\sigma(\alpha) \in \mathbb{R}$.
Then the existence of a positive lower bound for the height on $F^*\setminus\{\pm1\}$ follows from Bilu's famous equidistribution result \cite{Bi97}. Using a theorem of Garza we can achieve an effective constant, which we will prove in the following paragraph. 

Let $\beta$ be an arbitrary element in $F^*\setminus\{\pm 1\}$. We choose a number field $L/K$ with $L\subseteq K^{tr}$ and $\beta\in L(\alpha)$.
Since a $K$-embedding of $L(\alpha)$ is real if and only if it is an extension of a real $K$-embedding of $K(\alpha)$, we know that there are exactly $r_{\alpha,K}[L(\alpha):K(\alpha)]$ of those. Furthermore, every real $K$-embedding of $L(\alpha)$ is an extension of a real $K$-embedding of $K(\beta)$. 
Together this yields
\begin{equation}\label{eq:2}
R_{\beta,K}=\frac{r_{\beta,K}[L(\alpha):K(\beta)]}{[L(\alpha):K]} \geq \frac{r_{\alpha,K}[L(\alpha):K(\alpha)]}{[L(\alpha):K]}= R_{\alpha,K} \neq 0 \quad.
\end{equation}
Therefore, \eqref{eq:1} implies $R_{\beta,\mathbb{Q}}\geq R_{\alpha,K}[K:\mathbb{Q}]^{-1}=C_F$. Note that by \eqref{eq:2} the positive constant $C_F$ does not depend on the choice of the generator $\alpha$. Now we can use \cite{Ga07}, Theorem 1, to achieve
\begin{equation}\label{heightestimate}
h(\beta)\geq \frac{C_F}{2} \log\left( \frac{2^{1-\nicefrac{1}{C_F}}+\sqrt{4^{1-\nicefrac{1}{C_F}} +4}}{2}\right)=c_F > 0
\end{equation}
which is the claimed result.

Now we assume $F=K^{tr}(\alpha)$ to have no real $K$-embedding. Let $F'$ be the intersection of the normal hull of $F$ over $K$ with $\mathbb{R}$. Then $F'$ is a finite extension of $K^{tr}$ that has a real $K$-embedding. Hence we know from above that there exists an effective positive constant $c_{F'}$ with $h(\beta)\geq c_{F'}$ for all $\beta\in F'^*\setminus \{\pm 1\}$.

Assume there is a $\beta \in F^*$ with $h(\beta)< \frac{c_{F'}}{2}$. As we have seen above, for all $\sigma \in \Gal(\overline{\mathbb{Q}}/ K)$ the element $\sigma(\beta)$ is not in $\mathbb{R}$. Let $\overline{\beta}$ be the complex conjugate of $\beta$. Then $\beta\overline{\beta}$ is an element in $F'$. 
But from basic properties of the height $h$
we also have $h(\beta\overline{\beta}) \leq 2h(\beta)< c_{F'}$. This means that $\vert \beta \vert = \beta \overline{\beta}=1$. By construction of $F'$, the same is true for all $K$-conjugates of $\beta$. Hence, all $K$-conjugates of $\beta$ lie on the unit circle, proving the theorem with $c_F=\frac{c_{F'}}{2}$.
\end{proof}

\begin{corollary}\label{cor:2}
A finite extension $F$ of $K^{tr}$ has the Bogomolov property if and only if $i\not\in F$.
\end{corollary}

\begin{proof}
Let $F/K^{tr}$ be a finite extension. If $i\in F$ then we know that $\mathbb{Q}^{tr}(i) \subseteq F$, and hence $F$ does not have the Bogomolov property. If, on the other hand, $i\not\in F$ then Lemma \ref{lemma:1} yields that the only elements in $F$ whose $K$-conjugates all lie on the unit circle are $\pm 1$. By Theorem \ref{finextensions} we can conclude that $F$ has the Bogomolov property.
\end{proof}

\section{Applications}

Now we can provide positive answers to Questions \ref{question:2} and \ref{question:1}. We will give one explicit example for each question, these examples can be varied and extended in many ways.

First we will use Corollary \ref{cor:2} to answer Question \ref{question:2}.

\begin{example}\rm \label{answer:2}
Let $p$ be your favorite prime and set $F=\mathbb{Q}(\sqrt{p})^{tr}$. By Theorem \ref{finextensions} we know that $F$ has the Bogomolov property. However, a $4$-th root of $p$ is contained in $F$, and hence the Galois closure contains $i$ and does not have the Bogomolov property.
\end{example}

Next we will use Theorem \ref{finextensions} to construct a PAC field which has the Bogomolov property, and hence give a positive answer to Question \ref{question:1}. We say that a field $F$ is formally real if and only if $-1$ cannot be written as a sum of squares in $F$.

\begin{example}\rm \label{answer:1}
By Pop's result \cite{Pop96}, Theorem $\mathfrak{S}$, every finite extension of $\mathbb{Q}^{tr}$ which is not formally real is PAC (see \cite{JR94}, Section 7, for a short explanation).
By Corollary \ref{cor:2} all finite extensions of $\mathbb{Q}^{tr}$ which do not contain $i$ have the Bogomolov property. In order to find a PAC field which has the Bogomolov property, it remains to construct  a finite extension $F/\mathbb{Q}^{tr}$ with $i\not\in F$ and such that $F$ is not formally real. Let $F$ be the splitting field of $f(x)=x^5 +x^3 +1$ over $\mathbb{Q}^{tr}$. We claim that $F$ is of the prescribed form.

The Galois group of $f(x)\in\mathbb{Q}[x]$ is isomorphic to $S_5$. Hence, since $\mathbb{Q}^{tr}/\mathbb{Q}$ is Galois, the group $\Gal(F/\mathbb{Q}^{tr})$ is isomorphic to a nontrivial normal subgroup of $S_5$. Note that the roots of $f(x)$ are not totally real. Moreover, the discriminant of $f(x)$ is equal to $3233$ which is a square in $\mathbb{Q}^{tr}$. Therefore, $\Gal(F/\mathbb{Q}^{tr})$ is isomorphic to a subgroup of $A_5$, and hence $\Gal(F/\mathbb{Q}^{tr})\cong A_5$. As this is a simple group, there is no degree two extension of $\mathbb{Q}^{tr}$ contained in $F$. In particular, $i\not\in F$.

We denote the roots of $f(x)$ by $\alpha_1,\dots,\alpha_5$. The second Newton-Girard formula yields
$$\sum_{j=1}^{5} \left( \frac{\alpha_j}{\sqrt{2}} \right)^2 =-1 \quad.$$
The element $\frac{\alpha_j}{\sqrt{2}}$ is contained in $F$ for all $j\in\{1,\dots,5\}$. Thus $F$ is not formally real, proving the claim.
\end{example}

\begin{remark}\rm
Let $F$ be the field from Example \ref{answer:1}. The lower bound we achieve for the height on elements in $F^*\setminus\{\pm 1\}$ is very low indeed. The field $F'=F\cap \mathbb{R}$ has degree $\frac{\vert A_5 \vert}{2}=30$ and some analysis shows that $F'$ has exactly two real embeddings. Therefore, with the notation from  the proof of Theorem \ref{finextensions}, it is $C_{F'}=\frac{1}{15}$ and we get
$$h(\alpha)\geq \frac{1}{60}\log\left(\frac{2^{-14} + \sqrt{4^{-14}+4}}{2}\right) > \frac{1}{2000000} \text{ for all } \alpha \in F^* \setminus \{\pm 1\}.$$
\end{remark}

\textit{Acknowledgment:} The author would like to thank Dan Haran for answering a question on totally imaginary extensions of $\mathbb{Q}^{tr}$ not containing $i$. His answer led to the construction in Example \ref{answer:1}. This work was supported by the DFG-Projekt \textit{Heights and unlikely intersections} HA~6828/1-1.

\end{document}